\RequirePackage{fix-cm}
\documentclass[smallextended]{svjour3}       
\smartqed  
\usepackage{graphicx}
\usepackage{amsfonts}
\usepackage{amsmath}
\usepackage{cite}

\usepackage{booktabs}

\usepackage{epstopdf}
\usepackage[caption=false]{subfig}

\usepackage[numbers,sort&compress]{natbib}
\bibpunct[, ]{[}{]}{,}{n}{,}{,}
\makeatletter
\def\NAT@def@citea{\def\@citea{\NAT@separator}}
\makeatother

\usepackage{multirow}
\begin{document}

\title{A linesearch projection algorithm for solving equilibrium problems without monotonicity in Hilbert spaces}
\titlerunning{A linesearch projection algorithm for equilibrium problems without monotonicity}
\author{Lanmei Deng     \and
        Rong Hu   \and
        Yaping Fang 
}
\authorrunning{L. Deng, R. Hu and Y. Fang} 


\institute{Lanmei Deng \at
              College of Mathematics, Sichuan University, Chengdu, Sichuan, P.R. China\\
              \email{lanmeideng@126.com}           
           \and
           Rong Hu \at
             College of Applied Mathematics, Chengdu University of Information Technology, Chengdu, Sichuan, P.R. China\\
              \email{ronghumath@aliyun.com}
           \and
              Yaping Fang\at
              College of Mathematics, Sichuan University, Chengdu, Sichuan, P.R. China\\
              \email{ypfang@scu.edu.cn}           \\
           \emph{Corresponding author.}
}
\date{Received: date / Accepted: date}

\maketitle
\begin{abstract}
We propose a linesearch projection algorithm for solving non-monotone and non-Lipschitzian equilibrium problems in Hilbert spaces. It is proved that the sequence generated by the proposed algorithm converges strongly to a solution of the equilibrium problem under the assumption that the solution set of the associated Minty equilibrium problem is nonempty. Compared with existing methods, we do not employ Fej\'{e}r monotonicity in the strategy of proving the convergence. This comes from projecting a fixed point instead of the current point onto a subset of the feasible set at each iteration. Moreover, employing an Armijo-linesearch without subgradient has a great advantage in CPU-time. Some numerical experiments demonstrate the efficiency and strength of the presented algorithm.
\end{abstract}

\noindent

\section{Introduction}
\label{sec:1}
Let $\mathbb{H}$ be a Hilbert space with the inner product $\langle\cdot,\cdot\rangle$ and the norm $\|\cdot\|$. The strong and weak convergence are denoted by `$\rightarrow$' and `$\rightharpoonup$' in $\mathbb{H}$, respectively. Let $C$ be a nonempty closed convex subset of $\mathbb{H}$ and $f:C \times C\to \mathbb{R}$ be an equilibrium bifunction, i.e., $f(x,x)=0$ for all $x \in C$. The equilibrium problem associated with $f$ and $C$ in the sense of Muu, Blum and Oettli \cite{Muu-Oettli-EP1992,Blum-Oettli-EP1994} consists in finding a point $x^*\in C$ such that $$f(x^*,y)\geq0, \quad \forall y\in C,$$ which is also known as the Ky Fan inequality \cite{Ky-Fan-VI1972}. This problem and its solution set are denoted by $EP(f,C)$ and $S_E$, respectively. Associated with $EP(f,C)$, the Minty equilibrium problem consists in finding  a point $x^*\in C$ such that $$f(y,x^*)\leq0,\quad  \forall y\in C,$$which is denoted by $MEP(f,C)$ and its solution set is denoted by $S_M$. For an excellent survey on the existence results of $EP(f,C)$, the readers are referred to \cite{sm-se-EP2003, existence-EP2013} and the references quoted therein.

$EP(f,C)$ is convenient for reformulating the variational inequality problem, the fixed point problem, the saddle point problem, the optimization problem, the generalized Nash equilibrium problem in noncooperative game theory, the complementarity problem and the minimax problem (see, e.g., \cite{D.Q.Tran-Extragradient-EP2008, Regularization-VI2009, extragradient-VI1976, subgradient-extragradient-VI2011, NST-condition2007,HuangNJ-projection-VI2011} and the references therein). For instance, when $f(x, y)=\langle F(x), y-x\rangle$ with $F: C\to \mathbb{H}$ being a mapping, $EP(f,C)$ collapses to the variational inequality problem $VI(F,C)$ introduced by Lions and Stampacchia \cite{classicalVI1967}:
$$\text{ Find  }\ x^*\in C \, \text{ such  that } \,  \langle F(x^*), y-x^*\rangle\geq0,\quad \forall y\in C$$
and   $MEP(f,C)$ collapses to the Minty variational inequality problem $MVI(F,C)$:
$$\text{ Find  }\ x^*\in C \, \text{ such  that } \,  \langle F(x), y-x^*\rangle\geq0,\quad \forall y\in C.$$
Based on the reformulation, solution methods for solving aforementioned problems can usually be extended, with suitable modifications, to $EP(f,C)$.

The existing methods for solving $EP(f,C)$ can be mainly summarized as follows: the projection methods \cite{Regularization-VI2009, Relaxed-projection-FEP2011}, the extragradient methods with or without linesearches \cite{Anh-hybrid2013, Dual-extragradient-EP2012, D.Q.Tran-Extragradient-EP2008, combined-relaxation-EP2006,DENG-HU-FANG2020,Hieu}, the proximal point methods \cite{PPM-EP2010, SplittingPM-EP2009, PPM-NEP2003}, the subgradient methods \cite{inexact-subgradient-EP2011, subgradient-extragradient-EP2015}, the methods based on gap function \cite{Gap-function-EP2003, Iteration-Gap-function-EP2012}, and the methods based on auxiliary problem \cite{auxiliary-EP2003, proximal-like-Equi-programming1997}. Each of these solution methods is adapted to a class of equilibrium problems and guarantees the convergence of the method. See \cite{existence-EP2013} and the references quoted therein for an excellent survey on the existing methods.

Among the solution methods for solving $EP(f,C)$, the projection-type methods have a good advantage in implementing the iteration when the feasible set $C$ has a simple structure, such as a ball or a polyhedral set. This is the reason why they are popular with experts and researchers. To the best of our knowledge, most of projection-type methods require at least the assumption of pseudomonotonicity of the equilibrium function $f$ (see, e.g., \cite{projection-pseudo-EP2015, Anh-projection-pseudomonotone-EP2013,Hieu}). For example, Dinh and Muu \cite{projection-pseudo-EP2015} presented a projection algorithm for solving pseudomonotone and non-Lipschitzian equilibrium problems and analyzed the convergence. However, the pseudomonotonicity assumption may not be satisfied in some practical problems, for instance, the Nash-Cournot equilibrium problem considered in \cite{Dual-extragradient-EP2012}. In light of this situation, motivated by Dinh and Muu \cite{projection-pseudo-EP2015} and Ye and He \cite{Yiran-He-NVI2015}, Dinh and Kim \cite{projection-NEP2016} proposed projection algorithms for solving nonmonotone $EP(f,C)$. Their convergence does not require any monotonicity and Lipschitz-type property of the equilibrium bifunction $f$ but the nonemptyness of $S_M$. More precisely, Dinh and Kim \cite{projection-NEP2016} modified the second projection step of the projection algorithm for solving pseudomonotone and non-Lipschitzian equilibrium problems in \cite{projection-pseudo-EP2015}. The modification consists in projecting the current point onto shrinking convex subsets of $C$ which contain $S_M$. This modification guarantees the convergence without the assumption of pseudomonotonicity of the equilibrium function $f$. Their strategy is the same as the one of Strodiot, Vuong and Nguyen in \cite{J.J.Strodiot-Shrinking-pro-extra-NEP2016}. In particular, the algorithm of Dinh and Kim \cite{projection-NEP2016} coincides with the double projection algorithm of Ye and He \cite{Yiran-He-NVI2015} when the equilibrium problem reduces to a finite dimensional variational inequality problem.

Recently, Burachik and D\'{\i}az Mill\'{a}n \cite{Burachik-multiVI2019} presented a projection-type algorithm for solving nonmonotone variational inequality problems for point-to-set operators and obtained that the sequence generated by the proposed algorithm converges to a solution of the variational inequality. By projecting a fixed point instead of the current point onto a subset of the feasible set at each iteration, they proved the convergence without using Fej\'{e}r convergence, which is a classical tool for existing projection-type methods (see \cite{Yiran-He-NVI2015}).

 Inspired by Burachik and D\'{\i}az Mill\'{a}n \cite{Burachik-multiVI2019} and Dinh and Kim \cite{projection-NEP2016}, we propose a projection algorithm for solving nonmonotone and non-Lipschitzian equilibrium problems in $\mathbb{H}$ and prove that the sequence generated by the proposed algorithm converges strongly to a solution of $EP(f, C)$. More precisely, we take the technique in \cite{Burachik-multiVI2019,J.J.Strodiot-Shrinking-pro-extra-NEP2016} of projecting a fixed point instead of the current point onto a subset of the feasible set at each iteration. Compared with the method in \cite{J.J.Strodiot-Shrinking-pro-extra-NEP2016}, the difference consists in projecting a fixed point to different half-spaces at each iteration, allowing us to prove the convergence without using Fej\'{e}r monotonicity. In addition, compared with the method of Dinh and Kim \cite{projection-NEP2016},  we employ an Armijo-linesearch  without subgradient, which was introduced by \cite{D.Q.Tran-Extragradient-EP2008}. Our motivation originates in the result reported by Burachik and D\'{\i}az Mill\'{a}n \cite{Burachik-multiVI2019} that their strategy does well in solving nonmonotone variational inequality problems. Thus we extend the method for $VI(F, C)$ in finite dimensional spaces in \cite{Burachik-multiVI2019} to $EP(f, C)$ in Hilbert spaces, with suitable modifications. Compared with the methods in \cite{projection-NEP2016,Yiran-He-NVI2015,Burachik-multiVI2019}, our algorithm has a great advantage in the number of iterations and CPU-time. This has been illustrated by numerical experiments.

 The rest of the paper is organized as follows. In Sect. 2, we recall some preliminary results. Sect. 3 contributes to presentation of a projection algorithm for nonmonotone $EP(f,C)$ and its convergence under the nonemptyness of $S_M$. In Sect. 4, some numerical results demonstrate the strength and efficiency of the proposed algorithm. Finally, Sect. 5 makes a conclusion of the result of our work.
\section{Preliminaries}
\label{sec:2}
For each $x\in \mathbb{H}$, there exists a unique point in $C$, denoted by $P_C(x)$, 
such that $$\|x-P_C(x)\|\leq\|x-y\|,\quad \forall y\in C.$$ The mapping $P_C:\mathbb{H}\to C$ is known as the metric projection. Let $d(\cdot,C)$ be the distance function to $C$, i.e., $d(x,C)=\inf\{\|x-y\|:y\in C\}$. The following well-known results of the projection $P_C$ will be used in the sequel.
\begin{lemma}\label{projeciton-property} (See \cite{CAMOT2011})
The following statements hold:
\begin{enumerate}
\item[(i)]$z=P_C(x)$ if and only if $\langle x-z,y-z\rangle\leq0,\ \forall y\in C;$  
\item[(ii)]$\|P_C(x)-P_C(y)\|^2\leq\|x-y\|^2-\|P_C(x)-x+y-P_C(y)\|^2,\ \forall x,y\in C.$
\end{enumerate}
\end{lemma}
It is immediate from (ii) that $P_C$ is a nonexpansive mapping, i.e.,
$$\|P_C(x)-P_C(y)\|\leq\|x-y\|,\quad \forall x,y\in \mathbb{H}.$$

For solving $EP(f,C)$, one needs to consider some additional properties imposed on $f$ such as convexity and continuity. Now we recall some 
definitions.

\begin{definition}\label{lower-semicontinuous}(See \cite[ Definition 1.21]{CAMOT2011}) A function $f: \mathbb{H}\to [-\infty,+\infty]$ is said to be lower semicontinuous at $x\in \mathbb{H}$ if for every sequence $\{x^k\}$ converging strongly to $x$, it holds that $f(x)\leq\liminf_{k\rightarrow\infty}f(x^k)$. $f$ is upper semicontinuity at $x\in \mathbb{H}$ if $-f$ is lower semicontinuous at $x$. If $f$ is lower semicontinuous  and upper semicontinuous at $x\in \mathbb{H}$, then it is continuous at $x\in \mathbb{H}$. Furthermore, $f$ is continuous on $C$ if it is continuous at each $x\in C\subset\mathbb{H}$.
\end{definition}

\begin{definition}\label{continuity} (See \cite[Definition 2.1]{projection-NEP2016}) 
 A bifunction $f: C \times C \to \mathbb{R} $ is said to be jointly weakly continuous on $C \times C$ if  $x,y\in C$ and $\{x^k\}$ and $\{y^k\}$  are two sequences in $C$ converging weakly to $x$ and $y$, respectively, then $f(x^k,y^k)$ converges to $f(x,y)$.
\end{definition}

\begin{definition}\label{subdifferential} Let $f:\mathbb{H}\times \mathbb{H}\to \mathbb{R}$ be a function such that $f(x,\cdot)$ is convex for all $x \in \mathbb{H}$. For $x,y \in \mathbb{H}$, the subdifferential $\partial_2f(x,y)$ of $f(x,\cdot)$ at $y$ is defined by $$\partial_2f(x,y)=\{\xi\in\mathbb{H}:f(x,z)-f(x,y)\geq\langle\xi,z-y\rangle,\forall z\in \mathbb{H}\}.$$
\end{definition}

For solving $EP(f,C)$, we consider the following assumptions required in the sequel:
\begin{itemize}
 \item[(A1)]$f(x,\cdot)$ is convex on $C$ for all $x \in C$;
 \item[(A2)]$f$ is jointly weakly continuous on $C\times C$;
 \item[(A3)]$S_M\neq\emptyset$.
\end{itemize}

\begin{remark}\label{remark1}
The inclusion $S_M\subset S_E$ holds true when $f(\cdot ,y)$ is upper semicontinuous for each $y \in C$ and $f(x,\cdot)$ is convex and lower semicontinuous for each $x \in C$, while the converse inclusion $S_E\subset S_M$ can be guaranteed by the pseudomonotonicity of $f$ on $C$ (see \cite{sm-se-EP2003,existence-EP2013} and the references quoted therein). Thus assumptions (A1)-(A3) deduce that $S_M\subset S_E$ and $S_E\neq\emptyset$.
\end{remark}

The following lemmas will contribute to presenting the algorithm for solving nonmonotone $EP(f,C)$ and analysing the convergence of the proposed algorithm.

\begin{lemma}\label{lemma.9} (\cite[Proposition 2.1]{auxiliary-EP2003})
Under the assumption (A1), when $\rho>0$, a point $x^*\in C$ is a solution of $EP(f,C)$ if and only if it is a solution to the equilibrium problem:$$Find \ x^*\in C: f(x^*,y)+\frac{\rho}{2}\|y-x^*\|^2\geq0,\quad \forall y\in C.$$
\end{lemma}

\begin{remark}
(See \cite{DENG-HU-FANG2020})
The above equivalence is based on assumption (A1) and $\rho>0$. When at least one of them is not established, i.e., there exists at least one $x_0\in C$ such that $f(x_0,\cdot)$ is not convex on $C$ or $\rho<0$, Bigi and Passacantando \cite{Auxiliary-priciples-EP2017} analysed in detail the conditions that guarantee the aforementioned equivalence and advantages that this equivalence brings.
\end{remark}

The following lemma can be regarded as an infinite-dimension version of Theorem 24.5 in \cite{Rockafellar1970}.
\begin{lemma}\label{lemma.3} (\cite[Proposition 2.1]{J.J.Strodiot-Shrinking-pro-extra-NEP2016})
Under assumptions (A1) and (A2), for $\bar{x}$,$\bar{y}\in C$ and sequences $\{x^k\}$, $\{y^k\}$ in $C$ converging weakly to $\bar{x}$ and $\bar{y}$, respectively, it follows that for any $\varepsilon>0$, there exist $\eta>0$ and $k_\varepsilon\in \mathbb{N}$ such that $$\partial_2f(x^k,y^k)\subset \partial_2f(\bar{x},\bar{y})+\frac{\varepsilon}{\eta}B,$$ for every $k\geq k_\varepsilon$, where $B$ denotes the closed unit ball in $\mathbb{H}$.
\end{lemma}

\begin{lemma}\label{lemma.7} (\cite[Lemma 2.6]{projection-NEP2016})
Under assumptions (A1) and (A2), if $\{x^k\}\subset C$ is bounded, $\rho>0$, and $\{y^k\}$ is a sequence such that $$y^k=argmin\{f(x^k,y)+\frac{\rho}{2}\|y-x^k\|^2: y\in C\},$$ then $\{y^k\}$ is bounded.
\end{lemma}

\begin{lemma}\label{boundednessofxk}(\cite[Lemma 2.6]{Bello-VI2019})
Let S be a nonempty, closed and convex set. Take $x^{0}, x \in {\mathbb{H}}$. Assume that $x^{0} \notin S$ and that $S \subseteq W(x):=\left\{y \in \mathbb{H}:\left\langle y-x, x^{0}-x\right\rangle \leq 0\right\}$. Then,$x \in B\left[\frac{1}{2}\left(x^{0}+\bar{x}\right), \frac{1}{2} \rho\right]$, where $\bar{x}=P_{S}\left(x^{0}\right) $ and $\rho=\operatorname{d}\left(x^{0}, S\right)=\left\|x^{0}-P_{S}\left(x^{0}\right)\right\|$.
\end{lemma}

\begin{lemma} \label{xk1xksolution}(\cite[Proposition 2.12]{Burachik-multiVI2019})
Let $x^{0}, x \in \mathbb{H} $ and $W(x)=\{y \in \mathbb{H}:\langle y-x,x^0-x\leq0\rangle\}$, then it holds that $x=P_{W(x)}\left(x^{0}\right)$.
\end{lemma}

\begin{lemma}\label{whole-strong-convergence} (\cite[Lemma 1.5]{CQ-method-fixed-point-problem2006})
Let $C$ be a nonempty closed convex subset of $\mathbb{H}$, $\{x^k\}\subset \mathbb{H}$ and $u\in \mathbb{H}$. If any weak cluster point of $\{x^k\}$ belongs to $C$ and $$\|x^k-u\|\leq\|u-P_C(u)\|, \quad \forall k\in \mathbb{N},$$ then $x^k\rightarrow P_C(u)$.
\end{lemma}

\section{Algorithm for Nonmonotone $EP(f,C)$}
\label{sec:3}

Now, by combining and modifying the algorithms in \cite{projection-NEP2016,J.J.Strodiot-Shrinking-pro-extra-NEP2016,Burachik-multiVI2019},  we propose the following algorithm for solving nonmonotone and non-Lipschitzian  $EP(f,C)$:

{\bf Algorithm 1} \\
{\bf Step 0}\hspace{0.2cm} Given $ \{\beta_{k}\}_{k \in \mathbb{N}} \subset[\breve{\beta}, \hat{\beta}]$ \text
{ such that } $ 0<\breve{\beta} \leq \hat{\beta}<+\infty $. Take $x^0\in C$, choose parameters $\theta \in (0,1)$ and $\delta \in(0,1)$, and set $k=0$.\\
{\bf Step 1}\hspace{0.2cm}Solve the strongly convex program
\begin{equation} \label{strongly-convex-problem}
   \min\{f(x^k,y)+\frac{\beta_{k}}{2}\|y-x^k\|^2:y\in C\}
\end{equation}
to obtain its unique solution $y^k$. If $y^k=x^k$, stop.
Otherwise,do Step 2.\\
{\bf Step 2}\hspace{0.2cm}(Armijo-linesearch) Find $m_k$ as the smallest positive integer $m$ satisfying
\begin{equation}\label{Armijo-linesearch}
    \left\{
             \begin{array}{lr}
             z^{k,m}=(1-\theta^m)x^{k}+\theta^my^{k}, &  \\
             f(z^{k,m},y^k)\leq-\frac{\delta\beta_{k}}{2}\|x^k-y^k\|^2.&
             \end{array}
   \right.
\end{equation}
Set $\theta_k=\theta^{m_k}$, $z^k=z^{k,m_k}$. If $0\in \partial_2f(z^k,z^k)$, stop. Otherwise, go to Step 3. \\
{\bf Step 3}\hspace{0.2cm} Take $g^k\in \partial_2f(z^k,z^k)$ and set
\begin{equation} \label{Hkdefinition}
   H_k=\{x\in \mathbb{H}:\langle g^k,x-z^k\rangle \leq0\}, \ \ \tilde{H}_k=\bigcap_{j=0}^{j=k} {H_j}.
\end{equation}
Define
\begin{equation*}
   W(x^k)=\{x\in \mathbb{H}:\langle x-x^k,x^0-x^k\rangle \leq0\}.
\end{equation*}
{\bf Step 4}\hspace{0.2cm}Compute
\begin{equation}\label{xk1}
   x^{k+1}=P_{C\cap \tilde{H}_k\cap W(x^k)}(x^0).
\end{equation}
If $x^{k+1}=x^{k}$, stop. Otherwise, set $k=k+1$ and go to Step 1.

\begin{remark}
The Armijo-linesearch in Step 2 was introduced by Quoc et al. \cite{D.Q.Tran-Extragradient-EP2008} and it can be also considered as a modification of \cite[Linesearch 2]{J.J.Strodiot-Shrinking-pro-extra-NEP2016}.
\end{remark}

\begin{remark}
 According to Lemma \ref{lemma.9}, there exists a unique solution $y^k$ for the strongly convex program in (\ref{strongly-convex-problem}). Thus it induces from $x^k\neq y^k$ that
\begin{equation}\label{welldefine1}
f(x^k,y^k)+\frac{\beta_{k}}{2}\|y^k-x^k\|^2<0.
\end{equation}
Furthermore, $g^k\neq 0,\forall k\in \mathbb{N}$.
\end{remark}
Particularly, when $f(x, y)=\langle F(x), y-x\rangle$ with $F: C\rightarrow \mathbb{H}$ being a mapping, $EP(f,C)$ collapses to the variational inequality problem $VI(F,C)$. In the case, Algorithm 1 reduces to the following algorithm for solving nonmonotone $VI(F,C)$.

{\bf Algorithm 2} \\
{\bf Step 0}\hspace{0.2cm} Given $ \{\beta_{k}\}_{k \in \mathbb{N}} \subset[\breve{\beta}, \hat{\beta}]$ \text
{ such that } $ 0<\breve{\beta} \leq \hat{\beta}<+\infty $. Take $x^0\in C$, choose parameters $\theta \in (0,1)$ and $\delta \in(0,1)$, and set $k=0$.\\
{\bf Step 1}\hspace{0.2cm} Compute
\begin{equation*}
   y^{k}=P_{C}(x^k-\frac{1}{\beta_{k}}F(x^k)).
\end{equation*}
If $y^k=x^k$, stop. Otherwise, go to Step 2.\\
{\bf Step 2}\hspace{0.2cm}(Armijo-linesearch) Find $m_k$ as the smallest positive integer $m$ satisfying
\begin{equation*}
    \left\{
             \begin{array}{lr}
             z^{k,m}=(1-\theta^m)x^{k}+\theta^my^{k}, &  \\
             \langle F(z^{k,m}), y^k-z^{k,m}\rangle\leq-\frac{\delta}{2\beta_{k}}\|x^k-y^k\|^2.&
             \end{array}
   \right.
\end{equation*}
Set $\theta_k=\theta^{m_k}$, $z^k=z^{k,m_k}$. If $F(z^{k})=0$, stop. Otherwise, go to Step 3. \\
{\bf Step 3}\hspace{0.2cm} Take
\begin{equation*}
   H_k=\{x\in \mathbb{H}:\langle F(z^{k}),x-z^k\rangle \leq0\}, \ \ \tilde{H}_k=\bigcap_{j=0}^{j=k} {H_j}.
\end{equation*}
Define
\begin{equation*}
   W(x^k)=\{x\in \mathbb{H}:\langle x-x^k,x^0-x^k\rangle \leq0\}.
\end{equation*}
{\bf Step 4}\hspace{0.2cm}Compute
\begin{equation*}
   x^{k+1}=P_{C\cap \tilde{H}_k\cap W(x^k)}(x^0).
\end{equation*}
If $x^{k+1}=x^{k}$, stop. Otherwise, set $k=k+1$ and go to Step 1.

\begin{remark}
Algorithm 2 can be viewed as a modification of the  point-to-point version of Algorithm $F$ of Burachik and D\'{\i}az Mill\'{a}n \cite{Burachik-multiVI2019}, in which a different linesearch is used. As a comparison,  Algorithm 2 employs an Armijo-linesearch with norm $\|x^k-y^k\|$, improving the numerical behavior. This will be shown in Section \ref{sec:4}.
\end{remark}

Now we show the validity and convergence of Algorithm 1. 

\begin{lemma} \label{step2well-lemma.2}
Under the assumption $x^k\neq y^k,\forall k\in \mathbb{N}$, the linesearch in Algorithm 1 is well-defined in the sense that, for each $k\in \mathbb{N}$, there exists a positive integer $m>0$ satisfying the inequality in (\ref{Armijo-linesearch}). Furthermore, $H_k (\forall k\in \mathbb{N})$ is nonempty closed convex providing that $S_M\neq\emptyset$.
\end{lemma}

\begin{proof}
Firstly, under the assumption $x^k\neq y^k$, we verify that for each $k$ there exists a positive integer $m_0$ such that
$$f(z^{k,m_0},y^k)\leq-\frac{\delta\beta_{k}}{2}\|x^k-y^k\|^2.$$
Indeed, by contradiction, we suppose that for every positive integer $m$ and  $z^{k,m}=(1-\theta^m)x^{k}+\theta^my^{k}$, it holds that
$$f(z^{k,m},y^k)>-\frac{\delta\beta_{k}}{2}\|x^k-y^k\|^2.$$
Since $\{z^{k,m}\}$ converges strongly to $x^k$ as $m\rightarrow \infty$, taking the limit as $m\rightarrow \infty$, from the jointly weak continuity of $f$, we obtain that
\begin{equation} \label{(3.5)}
     f(x^{k},y^k)\geq-\frac{\delta\beta_{k}}{2}\|x^k-y^k\|^2.
\end{equation}
This combining with (\ref{welldefine1}) claims that
\begin{equation}
  -\frac{\delta\beta_{k}}{2}\|x^k-y^k\|^2\leq f(x^k,y^k)<-\frac{\beta_{k}}{2}\|x^k-y^k\|^2.
\end{equation}
The above inequality implies that $\delta>1$. This contradicts the fact that $\delta \in(0,1)$.
Consequently, the Armijo-linesearch is well-defined.

Now we show the nonemptyness of $H_k$. The assumption $S_M\neq\emptyset$ implies that for each $x^*\in S_M$, $f(y,x^*)\leq0,\forall y\in C$ holds true. Thus $f(z^k,x^*)\leq0,\forall k\in \mathbb{N}$. On the other side, from $g^k\in\partial_2f(z^k,z^k)$ and the convexity of $f(z^k,\cdot)$, we get $$f(z^k,y)-f(z^k,z^k)\geq\langle g^k,y-z^k\rangle, \quad \forall y\in C.$$
Therefore, it deduces from $f(z^k,x^*)\leq0,\forall k\in \mathbb{N}$ and $f(z^k,z^k)=0$ that
$$\langle g^k,z^k- x^*\rangle \geq f(z^k,z^k)-f(z^k,x^*)\geq 0,$$ i.e., $x^*\in H_k, \forall k\in \mathbb{N}.$
\end{proof}\qed


\begin{proposition}\label{lemma.10}
The sequences $\{x^k\}$ and $\{y^k\}$ generated by Algorithm 1 satisfy the property:
\begin{equation} \label{(pro3.1)}
   f(x^k,y)\geq f(x^k,y^k)+\beta_{k}\langle x^k-y^k,y-y^k\rangle, \quad \forall y\in C.
\end{equation}
In particular, taking $y=x^k$, we obtain $f(x^k,y^k)+\beta_{k}\|x^k-y^k\|^2\leq0, \forall k\in \mathbb{N}$. Furthermore, if $y^k=x^k$ for some $k\in \mathbb{N}$, then $x^k$ is a solution of $EP(f,C)$.
\end{proposition}

\begin{proof}
It follows from $y^k= argmin\{f(x^k,y)+\frac{\beta_{k}}{2}\|y-x^k\|^2:y\in C\}$ that
$$ 0\in \partial_2f(x^k,y^k)+\beta_{k}(y^k-x^k)+N_C(y^k),$$
where $N_C(y^k)$ is the normal cone of $C$ at $y^k$ defined by $$N_C(y^k)=\{\omega\in \mathbb{R}^n:\langle \omega,y-y^k\rangle\leq0, \forall y\in C\}.$$
Namely, there exist $\xi\in \partial_2f(x^k,y^k)$ and $\bar{\omega}\in N_C(y^k)$ such that $$0=\xi+\beta_{k}(y^k-x^k)+\bar{\omega}.$$

On the one hand, $\xi\in \partial_2f(x^k,y^k)$ implies that
\begin{equation} \label{lemma.10(1)}
   f(x^k,y)\geq f(x^k,y^k)+\langle \xi,y-y^k\rangle,\quad \forall y\in C.
\end{equation}
On the other hand, it induces from $\bar{\omega}\in N_C(y^k)$ that
\begin{equation} \label{lemma.10(2)}
   \langle\bar{\omega},y-y^k\rangle\leq0, \quad \forall y\in C.
\end{equation}
Noting that $\bar{\omega}=\beta_{k}(x^k-y^k)-\xi$ and combining with (\ref{lemma.10(2)}), we have
\begin{equation} \label{lemma.10(3)}
\langle\beta_{k}(x^k-y^k)-\xi,y-y^k\rangle\leq0,\quad \forall y\in C.
\end{equation}

From (\ref{lemma.10(1)}) and (\ref{lemma.10(3)}), we get
\begin{equation*} \label{lemma.10(4)}
   f(x^k,y)\geq f(x^k,y^k)+\langle \xi,y-y^k\rangle
   \geq f(x^k,y^k)+\beta_{k}\langle x^k-y^k,y-y^k\rangle,\quad \forall y\in C.
\end{equation*}
Taking $y=x^k$, we obtain
\begin{equation} \label{lemma.10(4)}
   f(x^k,y^k)+\beta_{k}\| x^k-y^k\|^2\leq0.
\end{equation}
Particularly, if $y^k=x^k$ for some $k\in \mathbb{N}$, by (\ref{(pro3.1)}), $x^k$ is a solution of $EP(f,C)$.
\end{proof}\qed

\begin{proposition} \cite[Lemma 4.1]{D.Q.Tran-Extragradient-EP2008}
If $0\in \partial_2f(z^k,z^k)$, then $z^k$ is a solution of $EP(f,C)$.
\end{proposition}

\begin{proposition}
If $x^{k+1}=x^k$, then $x^k$ is a solution of $EP(f,C)$.
\end{proposition}

\begin{proof}
If $x^{k+1}=P_{C\cap \tilde{H}_k\cap W(x^k)}(x^0)=x^k$, then $x^k \in \tilde{H}_k\subset{H}_k$.
On the one hand, it deduces from $x^k\in {H}_k=\{x\in \mathbb{H}:\langle g^k,x-z^k\rangle \leq0\}$ with $g^k\in \partial_2f(z^k,z^k)$ that $$\theta_k\langle g^k,x^k-y^k\rangle=\langle g^k,x^k-z^k\rangle \leq0.$$
Since $\theta_k\in (0,1)$, it follows that
\begin{equation} \label{xk1xkproof1}
\langle g^k,y^k-z^k\rangle=(1-\theta_k)\langle g^k,y^k-x^k\rangle \geq0.
\end{equation}
By virtue of the linesearch (\ref{Armijo-linesearch}),
we get that
\begin{equation} \label{xk1xkproof2}
f(z^{k},y^k)\leq \frac{-\delta\beta_{k}}{2}\| x^k-y^k\|^2\leq0.
\end{equation}
On the other hand, by definition of $g^k \in \partial_2f(z^k,z^k)$, we have $$\langle g^k,y^k-z^k\rangle \leq f(z^{k},y^k)-f(z^{k},z^k)=f(z^{k},y^k).$$
This result combining with (\ref{xk1xkproof1}) and (\ref{xk1xkproof2}) claims that
$$0\leq \langle g^k,y^k-z^k\rangle \leq f(z^{k},y^k)\leq \frac{-\delta\beta_{k}}{2}\| x^k-y^k\|^2\leq0$$
Hence, $x^k=y^k$. By Proposition \ref{lemma.10}, $x^k$ is a solution of $EP(f,C)$.
\end{proof}


From now on, we assume that $x^k\neq y^k, \forall k\in \mathbb{N}$ and thus $\{x^k\}$ generated by Algorithm 1 is infinite. We proceed to show the properties of  $\{x^k\}$.

\begin{proposition}\label{lemma.8}
Let $\bar{x}$ be a weak cluster point of the sequence $\{x^k\}$ generated by Algorithm 1 and $\{x^{k_j}\}$ be the corresponding subsequence converging  weakly to $\bar{x}$. We have $\bar{x}\in \tilde{H}_k\cap C, \forall k\in \mathbb{N}$ and thus $\bar{x}\in (\cap _{k=0}^\infty H_k)\cap C$.
\end{proposition}

\begin{proof}
By contradiction, we assume that there exists $k_0$ such that $\bar{x}\notin \tilde{H}_{k_0}$. The closedness and convexity of $\tilde{H}_{k_0}$ imply the weak closedness of $\tilde{H}_{k_0}$. It deduces from the weak closedness of $\tilde{H}_{k_0}$ that there exists $k_{1}>k_0$ such that $$x^{k}\notin \tilde{H}_{k_0}, \forall k\geq k_{1}.$$
Particularly, $x^{k_{1}}\notin \tilde{H}_{k_0}$. This contradicts the fact that $$x^{k_{1}}\in \tilde{H}_{k_{1}-1}\subset \tilde{H}_{k_{1}-2}\subset \cdots \subset \tilde{H}_{k_0}.$$
It follows that $\bar{x}\in \tilde{H}_k, \forall k\in \mathbb{N}$.

Moreover, it obtains from the closedness and convexity of $C$ that $C$ is weakly closed. Thus the closedness of $C$ implies that $\bar{x}\in C$. Consequently, we get the result that $$\bar{x}\in (\cap _{k=0}^\infty H_k)\cap C.$$
\end{proof}

\begin{remark}
Similar results  for different algorithms  have been established  in \cite[ Proposition 3.3, Proposition 3.4 and Proposition 3.7]{J.J.Strodiot-Shrinking-pro-extra-NEP2016} and \cite[Proposition 2]{DENG-HU-FANG2020}.
\end{remark}

\begin{proposition}\label{pro5}
Let $\tilde{H}_{k}$ be defined as in (\ref{Hkdefinition}) and set $\tilde{S}_{E}=\cap _{k=0}^\infty \tilde{H}_{k}\cap {S_E}$.
The following properties hold:
\begin{itemize}
\item[(i)]$\tilde{S}_{E}\neq\emptyset$;
\item[(ii)]$\tilde{S}_{E}\subset H_k \cap W(x^k), \forall k\in \mathbb{N};$
\item[(iii)]The sequence $\{x^k\}$ generated by Algorithm 1 is well-defined and $\{x^k\}\subset C$.
\end{itemize}
\end{proposition}

\begin{proof}
(i) It deduces from the definition of $\tilde{H}_{k}$ that
\begin{equation} \label{HK}
\cap _{k=0}^\infty \tilde{H}_{k}=\cap _{k=0}^\infty {H_k}.
\end{equation}
From Remark \ref{remark1}, we get that $S_M\subset S_E$. As shown in the proof of Lemma \ref{step2well-lemma.2}, it holds that $S_M\subset H_k, \forall k\in \mathbb{N}$. Combining with (\ref{HK}), we obtain that
\begin{equation} \label{nonempty}
S_M\subset \cap _{k=0}^\infty {H_k}\cap S_E=\cap _{k=0}^\infty \tilde{H}_{k}\cap S_E=\tilde{S}_{E}.
\end{equation}
The above inclusion and the assumption $S_M\neq\emptyset$ imply that $\tilde{S}_{E}\neq\emptyset$.

(ii) By definition, it holds that $\tilde{S}_{E}\subset H_k \cap C, \forall k\in \mathbb{N}$. By induction we proceed to verify that $\tilde{S}_{E}\subset W(x^k), \forall k\in \mathbb{N}$. For $k=0$, we have $\tilde{S}_{E}\subset W(x^0)=\mathbb{H}$. Suppose that
\begin{equation}\label{induction1}
\tilde{S}_{E}\subset W(x^k).
\end{equation}
By definition of $\tilde{S}_{E}$, we obtain that
\begin{equation}\label{induction2}
\tilde{S}_{E} \subset \tilde{H}_{k} \cap S_{E}.
\end{equation}
Hence, we deduce from (\ref{induction1}) and (\ref{induction2}) that
\begin{equation}\label{induction22}
\tilde{S}_{E} \subset \tilde{H}_{k} \cap W\left(x^{k}\right) \cap S_{E}\subset \tilde{H}_{k} \cap W\left(x^{k}\right) \cap C.
\end{equation}
By the part (i) of Lemma \ref{projeciton-property}, we obtain from $x^{k+1}=P_{C \cap \tilde{H}_{k} \cap W\left(x^{k}\right)}\left(x^{0}\right)$
that
\begin{equation}\label{induction3}
\left\langle x_{*}-x^{k+1}, x^{0}-x^{k+1}\right\rangle \leq 0,\quad \forall x_{*}\in \tilde{H}_{k} \cap W\left(x^{k}\right) \cap C.
\end{equation}
It follows from (\ref{induction22}) and (\ref{induction3}) that
\begin{equation}\label{induction4}
\left\langle x_{*}-x^{k+1}, x^{0}-x^{k+1}\right\rangle \leq 0,\quad \forall x_{*}\in \tilde{S}_{E}.
\end{equation}
By definition, the above inequality implies that $$\tilde{S}_{E}\subset W(x^{k+1}).$$ Consequently, $\tilde{S}_{E}\subset W(x^k), \forall k\in \mathbb{N}$.

(iii) Combing the above part (ii) with the fact that $S_E\subset C$, we have $$\tilde{S}_{E}\subset C\cap \tilde{H}_{k} \cap W(x^k), \forall k\in \mathbb{N}.$$ By virtue of the fact that $C$, $H_k$ and $W(x^k)$ ($\forall k\in \mathbb{N}$) are closed convex sets, the aforementioned part (i) claims that the closed convex set $C\cap H_k \cap W(x^k)\neq\emptyset, \forall k\in \mathbb{N}$ and thus $$C\cap \tilde{H}_{k} \cap W(x^k)\neq\emptyset, \forall k\in \mathbb{N}.$$ This implies that the projection step in (\ref{xk1}) is well-defined. The result that $x^{k} \in C, \forall k \in \mathbb{N}$ follows from $x^{0} \in C$ and the iterations in (\ref{xk1}).
\end{proof}

Now we show the boundedness of the sequences generated by Algorithm 1.
\begin{proposition} \label{boundedpro}
The sequence generated by  Algorithm 1 satisfies that $\{x^{k}\}\subset B\left[\frac{1}{2}\left(x^{0}+\bar{x}\right), \frac{\rho}{2}\right]$, where $\bar{x}:=P_{S_{M}}\left(x^{0}\right) $ and $\rho=\left\|x^{0}-P_{S_{M}}\left(x^{0}\right)\right\|$. Therefore, the sequence $\{x^{k}\}$
is bounded.
\end{proposition}

\begin{proof}
Since $S_{M}$ is a nonempty, convex and closed set and $x^{0} \notin S_{M}$, by Lemma \ref{boundednessofxk}, we get the result with setting $S=S_{M}$ and $x=x^{k}$, $\forall k \in \mathbb{N}$.
\end{proof}

\begin{corollary}\label{boundedsub-gk}
The sequences $\{y^{k}\}$ and $\{z^{k}\}$ generated by Algorithm 1 are bounded and $\{g^{k}\}$ admits a bounded subsequence.
\end{corollary}

\begin{proof}
By Proposition \ref{boundedpro}, $\{x^k\}$ is bounded. By Lemma \ref{lemma.7}, $\{y^k\}$ is bounded and so $\{z^k\}$ is bounded according to the definition of $z^k$.
Thus there exists a subsequence $\{z^{k_j}\}\subset \{z^{k}\}$ such that $z^{k_j}\rightharpoonup z^*$ as $j\rightarrow\infty$. By Lemma \ref{lemma.3},
$\{g^{k_j}\}$ is bounded.
\end{proof}

\begin{proposition}\label{boundedxkk}
The sequence $\{x^{k}\}$ generated by Algorithm 1 satisfies that
$\sum_{k=0}^{\infty}\left\|x^{k+1}-x^{k}\right\|^{2}<\infty$. Hence, $\lim _{k \rightarrow \infty}\left\|x^{k+1}-x^{k}\right\|=0$.
\end{proposition}

\begin{proof}
By Lemma \ref{xk1xksolution}, for all $k\in \mathbb{N}$, $x^k=P_{W(x^k)}(x^0)$ and thus $x^k\in W(x^k)$.
From the projection step (\ref{xk1}), $x^{k+1}\in W(x^k)$ holds true. Now we obtain from (ii) of Lemma \ref{projeciton-property} that $$0 \leq\left\|x^{k+1}-x^{k}\right\|^{2} \leq\left\|x^{k+1}-x^{0}\right\|^{2}-\left\|x^{k}-x^{0}\right\|^{2}.$$
Summing this inequality from $k=0$ to $\infty $ and employing the boundedness of the sequence $\{x^{k}\}$ obtained from
Proposition \ref{boundedpro}, we deduce that $$\sum_{k=0}^{\infty}\left\|x^{k+1}-x^{k}\right\|^{2}<\infty.$$The result that $\lim _{k \rightarrow \infty}\left\|x^{k+1}-x^{k}\right\|=0$ follows.
\end{proof}

Now we continue to show the convergence of Algorithm 1.
\begin{theorem}\label{THEOREM}
Let $\{x^k\}$ be the sequence generated by Algorithm 1 and $Cl\left(x^k\right)_{k\in \mathbb{N}}$
be the set of its weak cluster points. Then $Cl\left(x^k\right)_{k\in \mathbb{N}}\subset \tilde{S}_{E}\subset S_E$.
\end{theorem}

\begin{proof}
By Proposition \ref{lemma.8} and (\ref{HK}), we obtain that $Cl\left(x^k\right)_{k\in \mathbb{N}}\subset \cap _{k=0}^\infty H_k=\cap _{k=0}^\infty \tilde{H}_{k}$. By definition of $\tilde{S}_{E}$, it remains to show that
$Cl\left(x^k\right)_{k\in \mathbb{N}}\subset S_E$.
For each $x^* \in Cl\left(x^k\right)_{k\in \mathbb{N}}$, there exists a subsequence of $\{x^k\}$ (again denoted by $\{x^k\}$) converging weakly to $x^*$, i.e., $x^k\rightharpoonup x^*$ as $k\rightarrow\infty$. By virtue of the projection step (\ref{xk1}), we have $x^{k+1}\in H_k, \forall k\in\mathbb{N}$. It follows from the definition of $H_k$ in (\ref{Hkdefinition}) that
\begin{equation}\label{THEOREM1PRO1}
\left\langle g^{k}, x^{k+1}-z^{k}\right\rangle \leq 0, \forall k\in\mathbb{N}.
\end{equation}
On the other side,
\begin{equation}\label{THEOREM1PRO2}
\begin{aligned}
\left\langle g^{k}, x^{k+1}-z^{k}\right\rangle &=\left\langle g^{k}, x^{k+1}-x^{k}+x^{k}-z^{k}\right\rangle\\
&=\left\langle g^{k}, x^{k+1}-x^{k}\right\rangle+\left\langle g^{k}, x^{k}-z^{k}\right\rangle\\
&=\left\langle g^{k}, x^{k+1}-x^{k}\right\rangle+\theta_k\left\langle g^{k}, x^{k}-y^{k}\right\rangle,
\end{aligned}
\end{equation}
where the second equality is obtained from the definition of $z^{k}$ in (\ref{Armijo-linesearch}).
Thus combining (\ref{THEOREM1PRO1}) with (\ref{THEOREM1PRO2}), we obtain that
\begin{equation}\label{THEOREM1PRO3}
\theta_k\left\langle g^{k}, x^{k}-y^{k}\right\rangle\leq \left\langle g^{k}, x^{k}-x^{k+1}\right\rangle, \forall k\in\mathbb{N}.
\end{equation}
On the other hand, by Corollary \ref{boundedsub-gk}, there exists a bounded subsequence of $\{g^k\}$ (again denoted by $\{g^k\}$). Thus we deduce from Cauchy-Schwartz inequality and Proposition \ref{boundedxkk} that
\begin{equation}\label{THEOREM1PRO4}
\left\langle g^{k}, x^{k}-x^{k+1}\right\rangle\leq \|g^k\|\|x^{k}-x^{k+1}\|\rightarrow0,\quad \text{as}\ k\rightarrow\infty.
\end{equation}
Furthermore, according to the Algorithm 1, for all $k\in\mathbb{N}$, it follows from the definition of $g^k \in \partial_2f(z^k,z^k)$, the linesearch in (\ref{Armijo-linesearch}) and $x^k\neq y^k$ that
\begin{equation}\label{THEOREM1PRO5}
\begin{aligned}
(1-\theta_k)\left\langle g^{k}, y^{k}-x^{k}\right\rangle =\left\langle g^{k}, y^{k}-z^{k}\right\rangle &\leq f(z^{k},y^{k})-f(z^{k},z^{k})\\
&=f(z^{k},y^{k})\\
&\leq-\frac{\delta\beta_{k}}{2}\|x^k-y^k\|^2<0.
\end{aligned}
\end{equation}
This combining with $\theta_{k} \in(0,1)$ demonstrates that
\begin{equation}\label{THEOREM1PRO6}
\theta_k \left\langle g^{k}, x^{k}-y^{k}\right\rangle \geq \frac{\theta_k}{1-\theta_k}\frac{\delta\beta_{k}}{2}\|x^k-y^k\|^2>0.
\end{equation}
By virtue of (\ref{THEOREM1PRO3}), (\ref{THEOREM1PRO4}) and (\ref{THEOREM1PRO6}), we get
\begin{equation}
\frac{\theta_k}{1-\theta_k}\frac{\delta\beta_{k}}{2}\|x^k-y^k\|^2\rightarrow0,\quad \text{as}\ k\rightarrow\infty.
\end{equation}
Since $\{\beta_{k}\}$ is bounded and $\delta\in(0,1)$, it follows that
\begin{equation}\label{THEOREM1PRO7}
\frac{\theta_k}{1-\theta_k}\|x^k-y^k\|^2\rightarrow0,\quad \text{as}\ k\rightarrow\infty.
\end{equation}

Now we consider two distinct cases:

Case 1: $\limsup _{k\rightarrow \infty}\theta_{k}>0$. This deduces that there exist $\tilde{\theta}>0$ and a subsequence $\{\theta_{k_j}\}$ of the sequence $\{\theta_{k}\}$ such that $\theta_{k_j}>\tilde{\theta},\forall j\in \mathbb{N}$. On the one hand, according to (\ref{THEOREM1PRO7}), we obtain
$$\|x^{k_j}-y^{k_j}\|\rightarrow 0, \ as \ j\rightarrow\infty.$$
Combining this result with $x^{k_j}\rightharpoonup x^*$, we obtain $y^{k_j}\rightharpoonup x^*$.
From the definition of $y^{k_j}$, we get
\begin{equation*}\label{TH1case1(1)}
f(x^{k_j},y)+\frac{\beta_{k_j}}{2}\|y-x^{k_j}\|^2\geq f(x^{k_j},y^{k_j})+\frac{\beta_{k_j}}{2}\|y^{k_j}-x^{k_j}\|^2,\quad \forall y\in C.
\end{equation*}
On the other hand, by the boundedness of $\{\beta_{k}\}$, we may assume that there exists a subsequence of $\{\beta_{k_j}\}$ (again denoted by $\{\beta_{k_j}\}$) such that $\beta_{k_j}\rightarrow \tilde{\beta}>0$ as $j\rightarrow \infty$.

Taking the limit in the above inequality as  $j\rightarrow\infty$, from $x^{k_j}\rightharpoonup x^*$, $y^{k_j}\rightharpoonup x^*$, $\beta_{k_j}\rightarrow \tilde{\beta}>0$ and the jointly weak continuity of $f$, we have
\begin{equation*}\label{TH1case1(2)}
f(x^{*},y)+\frac{\tilde{\beta}}{2}\|y-x^{*}\|^2\geq f(x^*,x^*)=0,\quad \forall y\in C.
\end{equation*}
By Lemma \ref{lemma.9}, $x^*$ is a solution of $EP(f,C)$.

Case 2: $\lim _{k\rightarrow \infty}\theta_{k}=0$. By Corollary \ref{boundedsub-gk}, $\{y^{k}\}$ is bounded. Without loss of generality, we suppose that there exists a subsequence $\{y^{k_j}\}$ of $\{y^{k}\}$ such that $y^{k_j}\rightharpoonup y^*$ for some $y^* \in \mathbb{H}$ as $j\rightarrow \infty$. It is immediate from $\lim _{k\rightarrow \infty}\theta_{k}=0$ that $\lim _{j\rightarrow \infty}\frac{\theta_{k_j}}{\theta}=\lim _{j\rightarrow \infty} \theta^{m_{k_j}-1}=0$. On the one side, it follows from Proposition \ref{lemma.10} that
\begin{equation} \label{TH1case2.1}
   f(x^{k_j},y^{k_j})+\beta_{k_j}\|y^{k_j}-x^{k_j}\|^2\leq0.
\end{equation}

In addition, according to Algorithm 1, $m_{k_j-1}$ does not satisfy the Armijo-linesearch in (\ref{Armijo-linesearch}), i.e.,
\begin{equation} \label{TH1case2.2}
  f(z^{{k_j},m_{k_j-1}},y^{k_j})>-\frac{\delta\beta_{k_j}}{2}\|x^{k_j}-y^{k_j}\|^2, \quad \forall j\in \mathbb{N}.
\end{equation}

From (\ref{TH1case2.2}) and (\ref{TH1case2.1}), we obtain
\begin{eqnarray} \label{TH1case2.3}
\begin{aligned}
   f(z^{k_j,m_{k_j-1}},y^{k_j})&> -\frac{\delta\beta_{k_j}}{2}\|x^{k_j}-y^{k_j}\|^2\\
  &\geq \frac{\delta}{2}f(x^{k_j},y^{k_j}), \quad \forall j\in \mathbb{N}.
\end{aligned}
\end{eqnarray}

On the other side, according to Algorithm 1, $$z^{k_j,m_{k_j-1}}=(1-\theta^{m_{k_j-1}})x^{k_j}+\theta^{m_{k_j-1}}y^{k_j}.$$
Since $x^{k_j}\rightharpoonup x^*$, $y^{k_j}\rightharpoonup y^*$ and $\theta^{m_{k_j-1}}\rightarrow 0$, as $j\rightarrow \infty$, it holds that $z^{k_j,m_{k_j-1}}\rightharpoonup x^*$ as $j\rightarrow \infty$.
Moreover, the boundedness of $\{x^{k_j}\}$ and  $\{y^{k_j}\}$ induces that $\{\|x^{k_j}-y^{k_j}\|\}$ is also bounded. Hence, there exists a subsequence of $\{\|x^{k_j}-y^{k_j}\|\}$ (again denoted by $\{\|x^{k_j}-y^{k_j}\|\}$) converging to some $a\geq0$. In addition, by the boundedness of $\{\beta_{k}\}$, we may assume that there exists a subsequence $\{\beta_{k_j}\}$ such that $\beta_{k_j}\rightarrow \tilde{\beta}>0$ as $j\rightarrow \infty$.

Taking the limit in (\ref{TH1case2.3}), combining the continuity assumption with the result that $z^{k_j,m_{k_j-1}}\rightharpoonup x^*$, $x^{k_j}\rightharpoonup x^*$, $y^{k_j}\rightharpoonup y^*$, $\beta_{k_j}\rightarrow \tilde{\beta}>0$ and $\theta^{m_{k_j-1}}\rightarrow 0$ , we obtain that
\begin{equation*}
   f(x^*,y^*)\geq(-\frac{\tilde{\beta}}{2})a^2\geq \frac{\delta}{2} f(x^*,y^*).
\end{equation*}
This combining with $\delta\in(0,1)$ implies that $f(x^*,y^*)=0$ and $a=\lim_{j\rightarrow\infty}\|x^{k_j}-y^{k_j}\|=0$. By the case 1, $x^*$ is a solution of $EP(f,C)$.
Consequently, from the arbitrariness of $x^*$, we obtain $Cl\left(x^k\right)_{k\in \mathbb{N}}\subset S_E$. The result that $Cl\left(x^k\right)_{k\in \mathbb{N}}\subset \tilde{S}_{E}\subset S_E$ follows.
\end{proof}

\begin{theorem}\label{theorem3}
The sequence $\{x^k\}$ generated by Algorithm 1 converges strongly to a solution $x^*=P_{\cap _{k=0}^\infty {W(x^k)}}(x^0)$ of $EP(f,C)$ providing the assumptions (A1)-(A3).
\end{theorem}

\begin{proof}
By Theorem \ref{THEOREM}, $Cl\left(x^k\right)_{k\in \mathbb{N}}\subset \tilde{S}_{E}\subset {S}_{E}$ holds true, where $Cl\left(x^k\right)_{k\in \mathbb{N}}$ is the set of weak cluster points of $\{x^k\}$. Moreover, it follows from Proposition \ref{pro5} that $\tilde{S}_{E}\subset H_k \cap W(x^k), \forall k\in \mathbb{N}$. Hence we obtain that
$Cl\left(x^k\right)_{k\in \mathbb{N}}\subset W(x^k), \forall k\in \mathbb{N}$ and thus
\begin{equation*}
Cl\left(x^k\right)_{k\in \mathbb{N}}\subset \cap _{k=0}^\infty {W(x^k)}.
\end{equation*}
By Lemma \ref{whole-strong-convergence}, it is enough to prove that
\begin{equation} \label{strongconvergenceproof}
\|x^k-x^0\|\leq \|x^0-P_{\cap _{k=0}^\infty {W(x^k)}}(x^0)\|,\quad \forall k\in \mathbb{N}.
\end{equation}
Indeed, by Lemma \ref{xk1xksolution}, for all $k\in \mathbb{N}$, $x^k=P_{W(x^k)}(x^0)$.
Thus using Lemma \ref{projeciton-property}, we have
\begin{equation}
\|x^k-x^0\|\leq \|x-x^0\|, \quad \forall x\in W(x^k).
\end{equation}
Hence the fact that $P_{\cap _{k=0}^\infty {W(x^k)}}(x^0)\in \cap _{k=0}^\infty {W(x^k)}$ and the inclusion $\cap _{k=0}^\infty {W(x^k)} \subset W(x^k), \forall k\in \mathbb{N}$ imply the result in (\ref{strongconvergenceproof}).
\end{proof}

\begin{remark}
Under same assumptions, some algorithms been proposed in  \cite{PPM-NEP2003,projection-NEP2016,J.J.Strodiot-Shrinking-pro-extra-NEP2016,DENG-HU-FANG2020} to solve $EP(f,C)$. It is worth mentioning that Fej\'{e}r  convergence was used as a common tool  in  \cite{PPM-NEP2003,projection-NEP2016,J.J.Strodiot-Shrinking-pro-extra-NEP2016,DENG-HU-FANG2020}  for proving the convergence of algorithms. Compared with these existing results, we prove the convergence of Algorithm 1  without using Fej\'{e}r convergence.
\end{remark}

\begin{theorem}\label{theorem3-VI}
The sequence $\{x^k\}$ generated by Algorithm 2 converges strongly to a solution $x^*=P_{\cap _{k=0}^\infty {W(x^k)}}(x^0)$ of $VI(F,C)$ under the following assumptions:
\begin{itemize}
\item[(i)]  $x_k\rightharpoonup x\ $ implies  $F(x_k) \rightarrow F(x)$.
\item[(ii)] the solution set of $MVI(F,C)$ is nonempty.
\end{itemize}
\end{theorem}
\begin{proof}
Define $f:C\times C\to \mathbb{R}$ by $f(x,y)=\langle F(x),y-x\rangle$ for all $x,y\in C$.  By conditions (i) and (ii), the assumptions (A1)-(A3) are satisfied. So, the conclusion follows directly from Theorem \ref{theorem3}.
\end{proof}

\section{Numerical examples}\label{sec:4}
In this section, we employ Example 1 to illustrate the efficiency of the proposed Algorithm 1 and compare with \cite[Algorithm 2]{projection-NEP2016} (denoted by PA) of Dinh and Kim. Example 2 shows the efficiency of our Algorithm 2, \cite[Algorithm F]{Burachik-multiVI2019} (denoted by PA-BM) and \cite[Algorithm 2.1]{Yiran-He-NVI2015} (denoted by PA-YH). We implement the numerical experiments in MATLAB Version 7.0.0.19920 (R14) running on a Laptop with Intel(R) Core(TM) i5-2450M CPU $@$ 2.50 GHz with 6 GB RAM. We take $E(x^k)=\|x^{k}-z^k\|^2 \leq 10^{-4}$ and $E(x^k)=\|x^{k}-y^k\|^2 \leq 10^{-8}$ as the termination criteria for Example 1 and Example 2, respectively.

{\bf Example 1} Consider the equilibrium problem based on Nash-Cournot oligopolistic equilibrium models of electricity markets (see, e.g., \cite{Dual-extragradient-EP2012, projection-NEP2016, J.J.Strodiot-Shrinking-pro-extra-NEP2016}). We assume that there are $n^{c}$ (here we take $n^{c}=3$ ) companies and every company $i$ $(i = 1, 2, 3)$ has $I_i$ generating units (here, we take $I_1 = \{1\}$, $I_2 = \{2, 3\}$ and $I_3 = \{4, 5, 6\}$). We take $n^{g}$  (here, $n^{g}=6$) to be the number of all generating units and $x$ to be the vector whose entry $x_i$ stands for the power generating by unit $i$. As the strategies in \cite{Dual-extragradient-EP2012, projection-NEP2016, J.J.Strodiot-Shrinking-pro-extra-NEP2016}, we suppose that the price $p$ is a decreasing affine function of $\sigma=\sum_{i=1}^{n^{g}} x_{i}$. Hence,
$$p(x)=378.4-2 \sum_{i=1}^{n^{g}} x_{i}=p(\sigma).$$

Define the profit made by company $i$ as
$$f_{i}(x)=p(\sigma) \sum_{j \in I_{i}} x_{j}-\sum_{j \in I_{i}} x_{j} c_{j}\left(x_{j}\right),$$
where $c_{j}\left(x_{j}\right)$ stands for the cost for generating $ x_{j}$ and given by $$c_{j}\left(x_{j}\right) =\max \left\{c_{j}^{0}\left(x_{j}\right), c_{j}^{1}\left(x_{j}\right)\right\}, \quad j=1,2, \ldots, n^{g},$$
with
$$c_{j}^{0}\left(x_{j}\right) =\frac{\alpha_{j}^{0}}{2} x_{j}^{2}+\beta_{j}^{0} x_{j}+\gamma_{j}^{0}, \quad c_{j}^{1}\left(x_{j}\right) =\alpha_{j}^{1} x_{j}+\frac{\beta_{j}^{1}}{\beta_{j}^{1}+1} \gamma_{j}^{-1 / \beta_{j}^{1}}\left(x_{j}\right)^{\left(\beta_{j}^{1}+1\right) / \beta_{j}^{1}},$$
where $\alpha_{j}^{k}, \beta_{j}^{k}, \gamma_{j}^{k}(k=0,1)$ are given parameters in Table 1.

\begin{table}[!htbp]
\begin{center}{\bf Table 1} The parameters of cost functions of the generating units
\end{center}
\label{tab:1}
\centering
\begin{tabular}{lllllll}
\hline\noalign{\smallskip}
Gen. & $\alpha_{j}^{0}$ & $\beta_{j}^{0}$ & $\gamma_{j}^{0}$ & $\alpha_{j}^{1}$ & $\beta_{j}^{1}$ & $\gamma_{j}^{1}$  \\
\noalign{\smallskip}\hline\noalign{\smallskip}
1 & 0.0400 & 2.00 & 0.00 & 2.00 & 1.00 & 25.0000\\
2 & 0.0350 & 1.75 & 0.00 & 1.75 & 1.00 & 28.5714\\
3 & 0.1250 & 1.00 & 0.00 & 1.00 & 1.00 & 8.0000\\
4 & 0.0116 & 3.25 & 0.00 & 3.25 & 1.00 & 86.2069\\
5 & 0.0500 & 3.00 & 0.00 & 3.00 & 1.00 & 20.0000\\
6 & 0.0500 & 3.00 & 0.00 & 3.00 & 1.00 & 20.0000\\
\noalign{\smallskip}\hline
\end{tabular}
\end{table}

\begin{table}[!htbp]
\begin{center}{\bf Table 2} The lower and upper bounds of the power \\ generation of the generating units and companies
\end{center}
\label{tab:2}
\centering
\begin{tabular}{llllll}
\hline\noalign{\smallskip}
Com. & Gen. & $x_{\min }^{g}$ & $x_{\max }^{g}$ & $x_{\min }^{c}$ & $x_{\max }^{c}$ \\
\noalign{\smallskip}\hline\noalign{\smallskip}
1 & 1 & 0 & 80 & 0 & 80\\
2 & 2 & 0 & 80 & 0 & 130\\
2 & 3 & 0 & 50 & 0 & 130\\
3 & 4 & 0 & 55 & 0 & 125\\
3 & 5 & 0 & 30 & 0 & 125\\
3 & 6 & 0 & 40 & 0 & 125\\
\noalign{\smallskip}\hline
\end{tabular}
\end{table}

\begin{table}[!htbp]
\begin{center}{\bf Table 3} Results for Algorithm 1 and PA with $\theta$ \\ ($\delta=0.01$ and $\beta_k=0.5, \forall k \in \mathbb{N}$)
\end{center}
\label{tab:3}
\centering
\begin{tabular}{lllll}
\toprule
\multirow{2}{*}{$\theta$} & \multicolumn{2}{c}{Number of iterations} & \multicolumn{2}{c}{CPU-time(s)} \\
\cmidrule(r){2-3} \cmidrule(r){4-5} \\
&  Algorithm 1     &  PA
&  Algorithm 1     &  PA \\
\midrule
0.05  &   836    &    1661   &  41.8707  & 732.9239  \\
0.1   &   546    &     734   &  24.3050  &  99.7314  \\
0.2   &   322    &     431   &  13.8061  &  45.3339   \\
0.25  &   286    &     346   &  18.9073  &  42.1359  \\
0.5   &   150    &     263   &   9.4381  &  32.6042  \\
0.6   &   162    &     207   &   4.4148  &  12.8857 \\
0.7   &   171    &     191   &   4.5864  &  12.8545  \\
0.85  &   183    &     149   &   4.6020  &   8.2369  \\
0.95  &   187    &     184   &   4.9608  &  12.1681  \\
0.99  &   222    &     159   &   7.3944  &  15.1945 \\

\bottomrule
\end{tabular}
\end{table}

\begin{table}[!htbp]
\begin{center}
{\bf Table 4} Results for Algorithm 1 and PA with different $\beta_k$ \\
($\theta=0.1$ and $\delta=0.01$)
\end{center}
\label{tab:4}
\centering
\begin{tabular}{lllll}
\toprule
\multirow{2}{*}{$\beta_k$} & \multicolumn{2}{c}{Number of iterations} & \multicolumn{2}{c}{CPU-time(s)} \\
\cmidrule(r){2-3} \cmidrule(r){4-5} \\
&  Algorithm 1     &  PA
&  Algorithm 1     &  PA \\
\midrule

0.1 &   688  &  998  &   67.1116 & 281.2386 \\
0.2 &   617  &  966  &   33.6650 & 215.6090 \\
0.3 &   560  &  840  &   26.2082 & 134.1921  \\
0.4 &   548  &  755  &   22.8853 & 103.7563 \\
0.5 &   546  &  734  &   22.1365 &  94.7706\\
0.6 &   530  &  728  &  22.3081  &   90.5274 \\
0.7 &   530  &  729  &  20.0461  &   90.7302 \\
0.8 &   521  &  732  &  19.2817  &  91.1826 \\
0.9 &   518  &  737  &  18.9853  &   91.0266\\
1.0 &   508  &  742  &  19.1725  &   94.7394\\

\bottomrule
\end{tabular}
\end{table}

\begin{table}[!htbp]
\begin{center}
{\bf Table 5} Results for Algorithm 1 with different $\beta_k$ \\ ($\theta=0.1$ and $\delta=0.01$)
\end{center}
\label{tab:5}
\centering
\begin{tabular}{lll}
\hline\noalign{\smallskip}
$\beta_k$ & Number of iterations & CPU-time(s)\\
\noalign{\smallskip}\hline\noalign{\smallskip}

$\frac{k+1}{k+3}$     &      56       &     4.8048  \\
$\frac{k+1}{2k+3}$    &     43        &     1.7160   \\
$\frac{k+1}{3k+3}$    &     37        &      1.4976   \\
$\frac{k+1}{4k+3}$    &     33        &      1.3728  \\
$\frac{k+1}{5k+3}$    &     30        &      1.1388  \\
0.1                    &    688        &       64.7092  \\
0.3                    &    560        &       42.6819  \\
0.5                    &    546        &      39.0627   \\
0.7                    &     530       &      36.5354   \\
1.0                    &    508        &      33.6182   \\

\noalign{\smallskip}\hline
\end{tabular}
\end{table}

\begin{table}[!htbp]
\begin{center}
{\bf Table 6} Results for Algorithm 1 with different $\theta$ \\ ($\delta=0.01$ and $\beta_k=0.5, \forall k \in \mathbb{N}$)
\end{center}
\label{tab:6}
\centering
\begin{tabular}{lll}
\hline\noalign{\smallskip}
$\theta$ & Number of iterations & CPU-time(s)\\
\noalign{\smallskip}\hline\noalign{\smallskip}

0.1 &  546   &  40.8411 \\
0.3 &  249   &  17.2381 \\
0.5 &  150   &  9.5005  \\
0.8 &  175   &  10.7485  \\
0.99 &  222  &   15.0853  \\

\noalign{\smallskip}\hline
\end{tabular}
\end{table}

\begin{table}[!htbp]
\begin{center}
{\bf Table 7} Results for Algorithm 1 with different $\delta$ \\ ($\theta=0.1$ and $\beta_k=0.5, \forall k \in \mathbb{N}$)
\end{center}
\label{tab:7}
\centering
\begin{tabular}{lll}
\hline\noalign{\smallskip}
$\delta$ & Number of iterations & CPU-time(s)\\
\noalign{\smallskip}\hline\noalign{\smallskip}

0.01 &  546  &  24.8822 \\
0.05 &  546  &  20.3893 \\
0.1  &  546  &  21.7309 \\
0.25 &  546  &  20.0773 \\
0.5  &  546  &  19.2037 \\

\noalign{\smallskip}\hline
\end{tabular}
\end{table}

\begin{figure}[!htbp]
\centerline{
\includegraphics[scale=0.65]{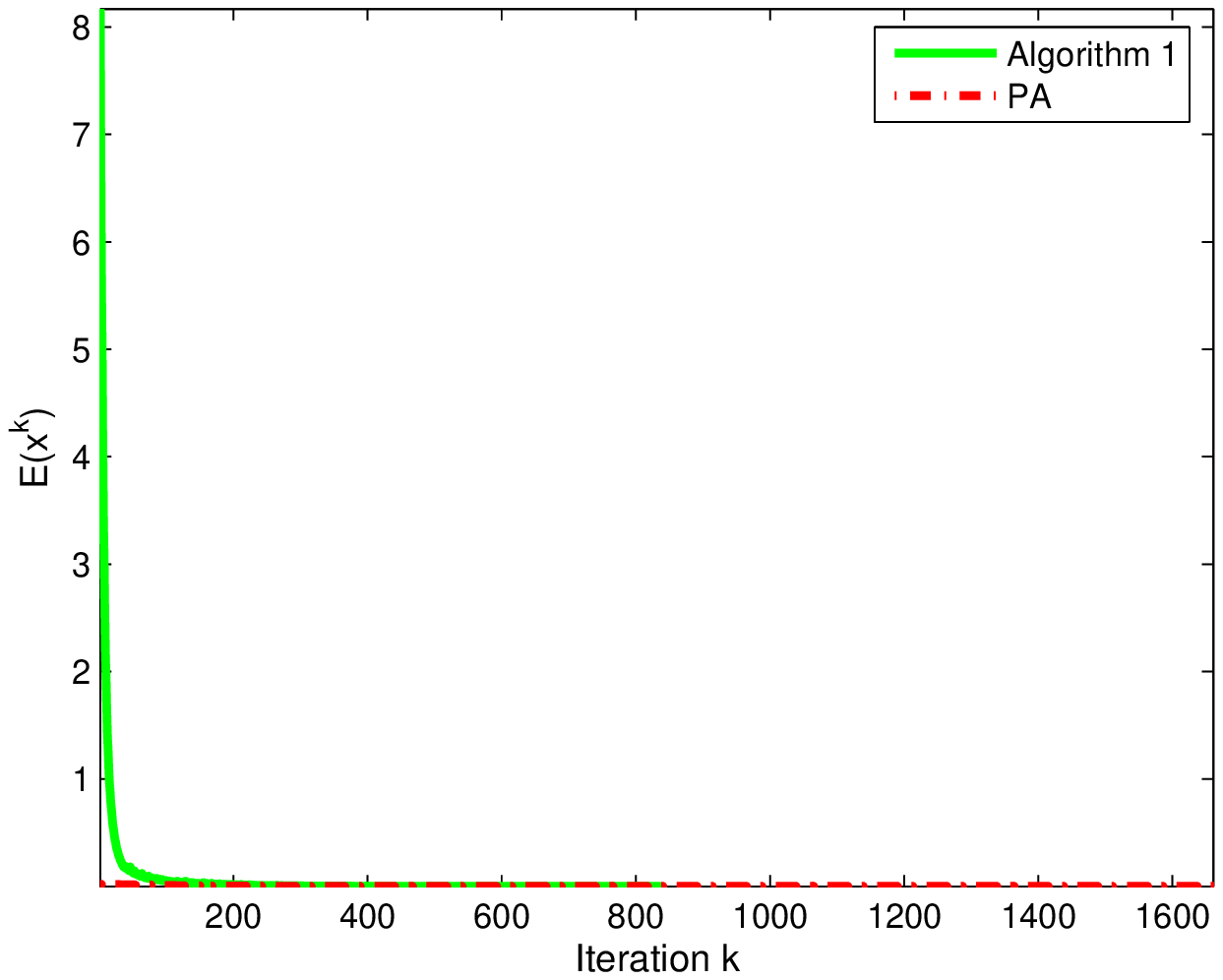}}
\centerline{{\bf Fig. 1} Results for Algorithm 1 and PA with $\theta=0.05, \delta=0.01$ and $\beta_k=0.5, \forall k \in \mathbb{N}$}
\label{fig:1}
\end{figure}

\begin{figure}[!htbp]
\centerline{
\includegraphics[scale=0.65]{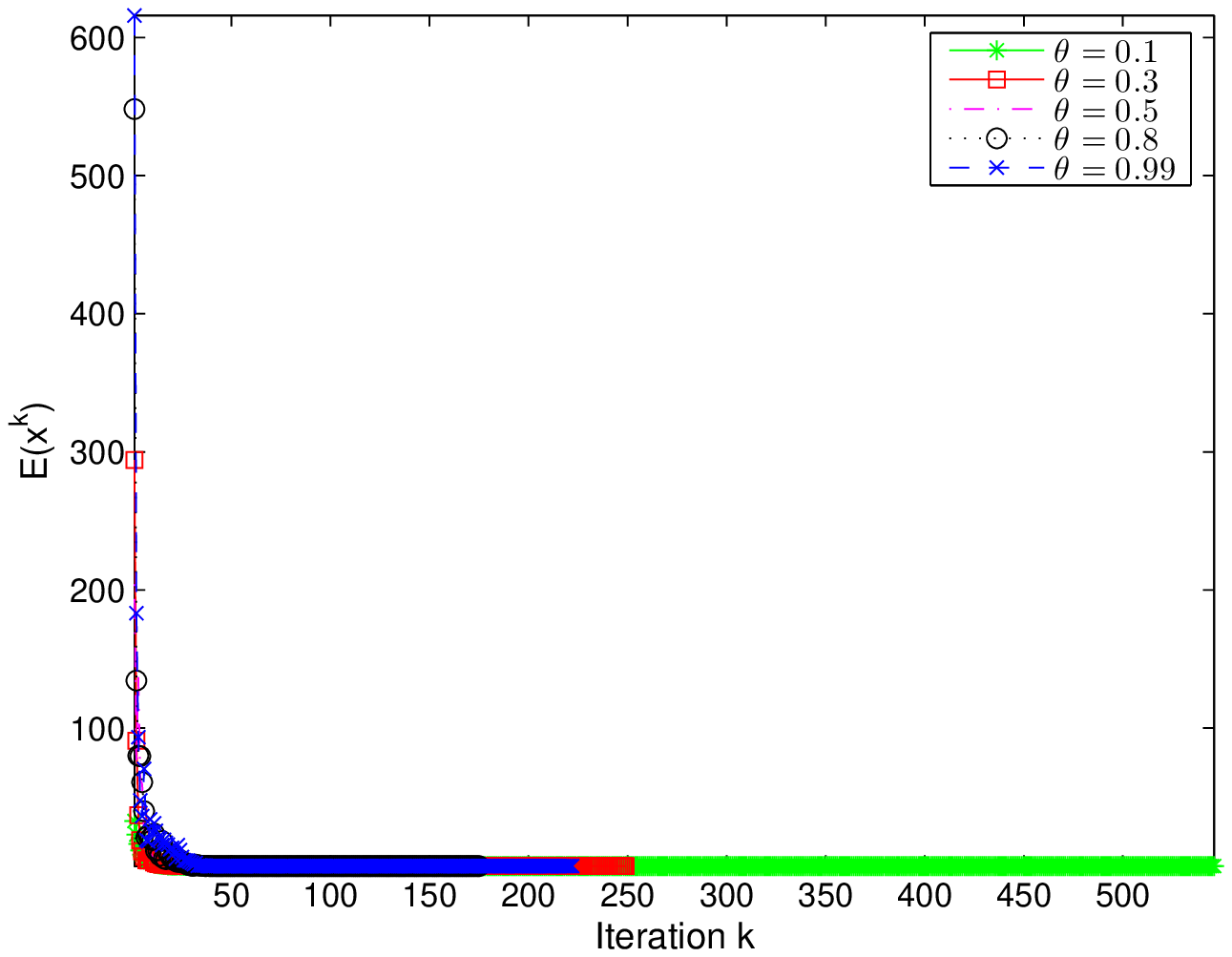}
}
\centerline{{\bf Fig. 2} Results for Algorithm 1 with different $\theta$ ($\delta=0.01$ and $\beta_k=0.5, \forall k \in \mathbb{N}$)}
\label{fig:2}
\end{figure}

\begin{figure}[!htbp]
\centerline{
\includegraphics[scale=0.65]{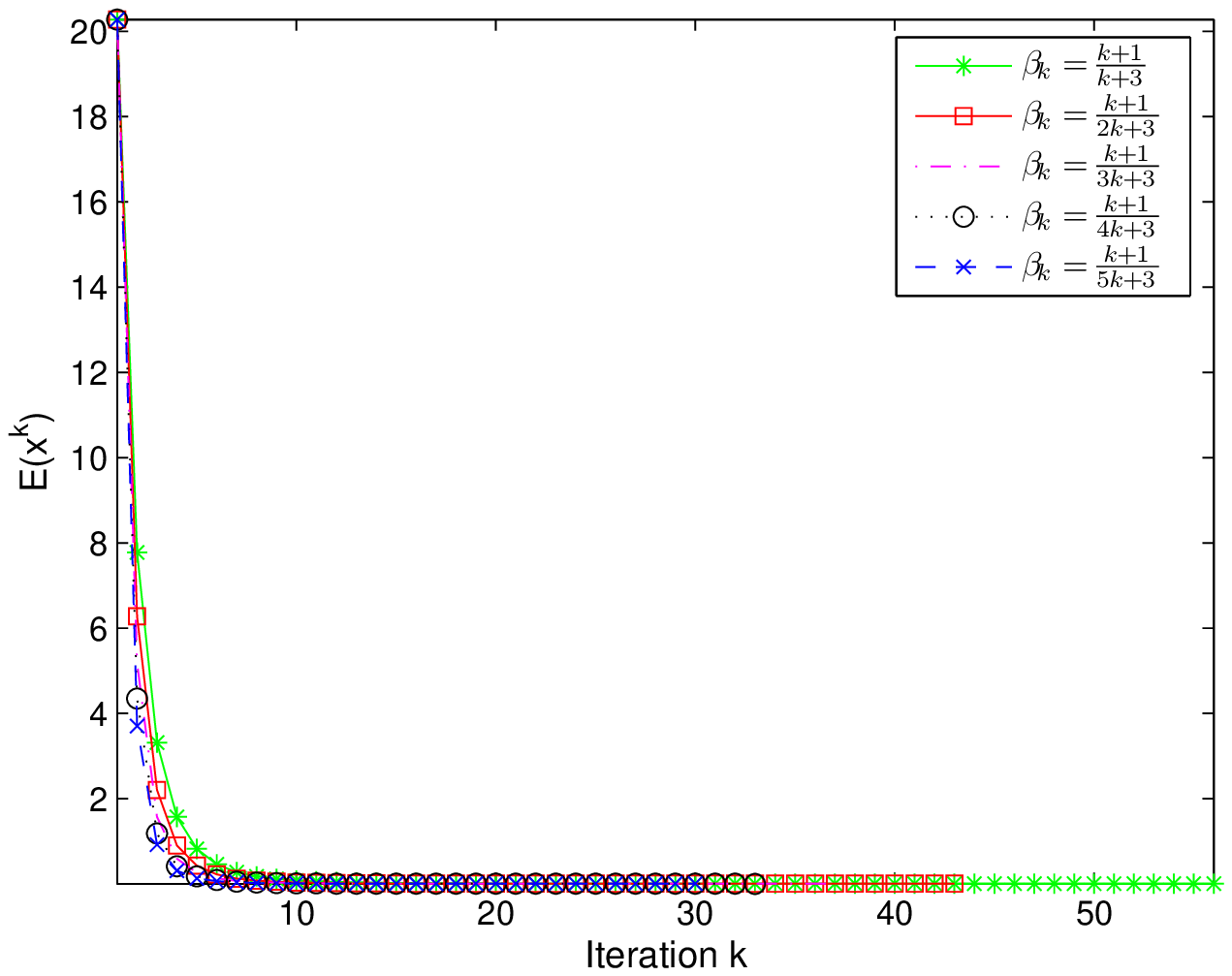}
}
\centerline{{\bf Fig. 3} Results for Algorithm 1 with different $\beta_k$ ($\theta=0.1$ and $\delta=0.01$)}
\label{fig:3}
\end{figure}

\begin{figure}[!htbp]
\centerline{
\includegraphics[scale=0.65]{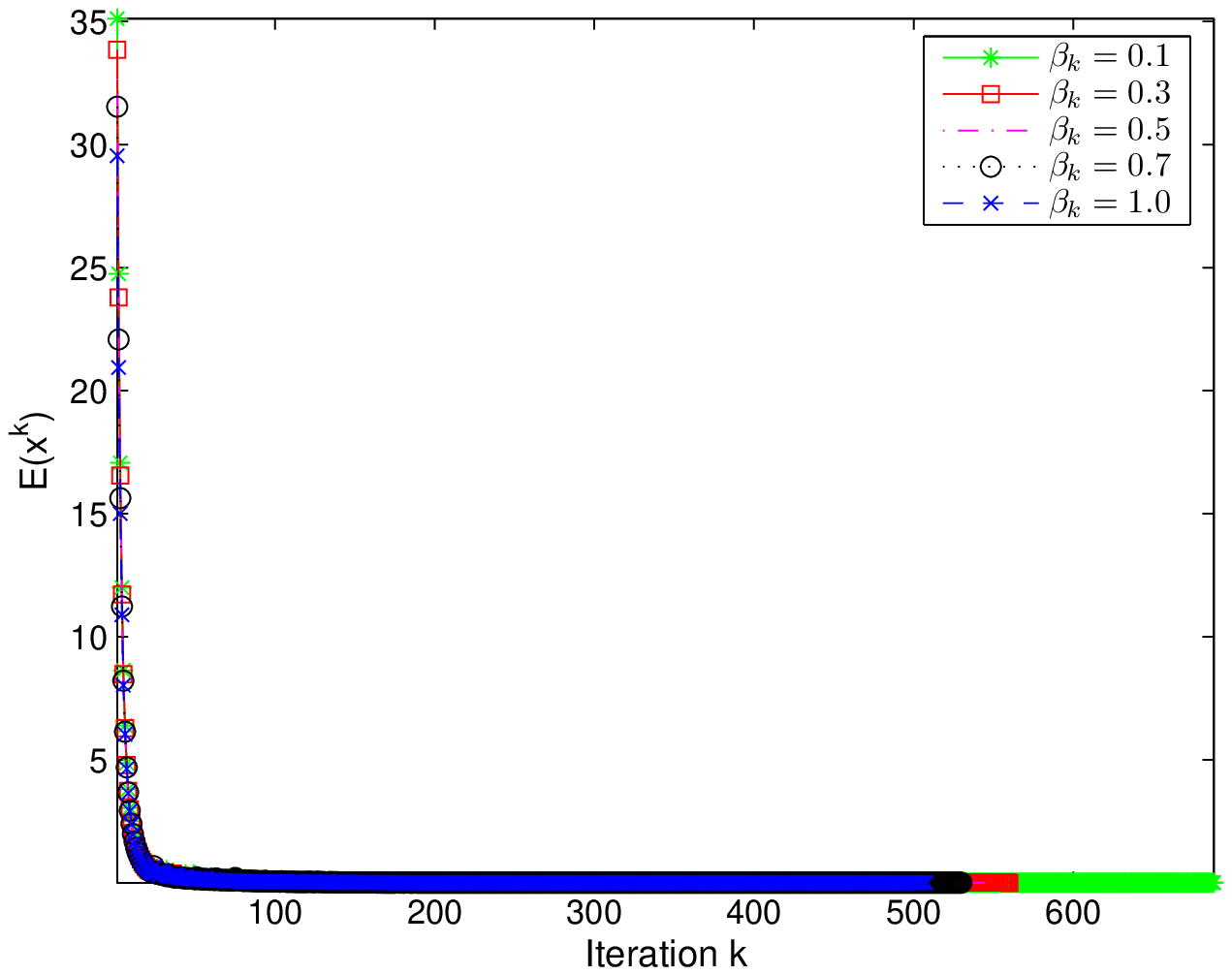}
}
\centerline{{\bf Fig. 4} Results for Algorithm 1 with different $\beta_k$ ($\theta=0.1$ and $\delta=0.01$)}
\label{fig:4}
\end{figure}

\begin{figure}[!htbp]
\centerline{
\includegraphics[scale=0.65]{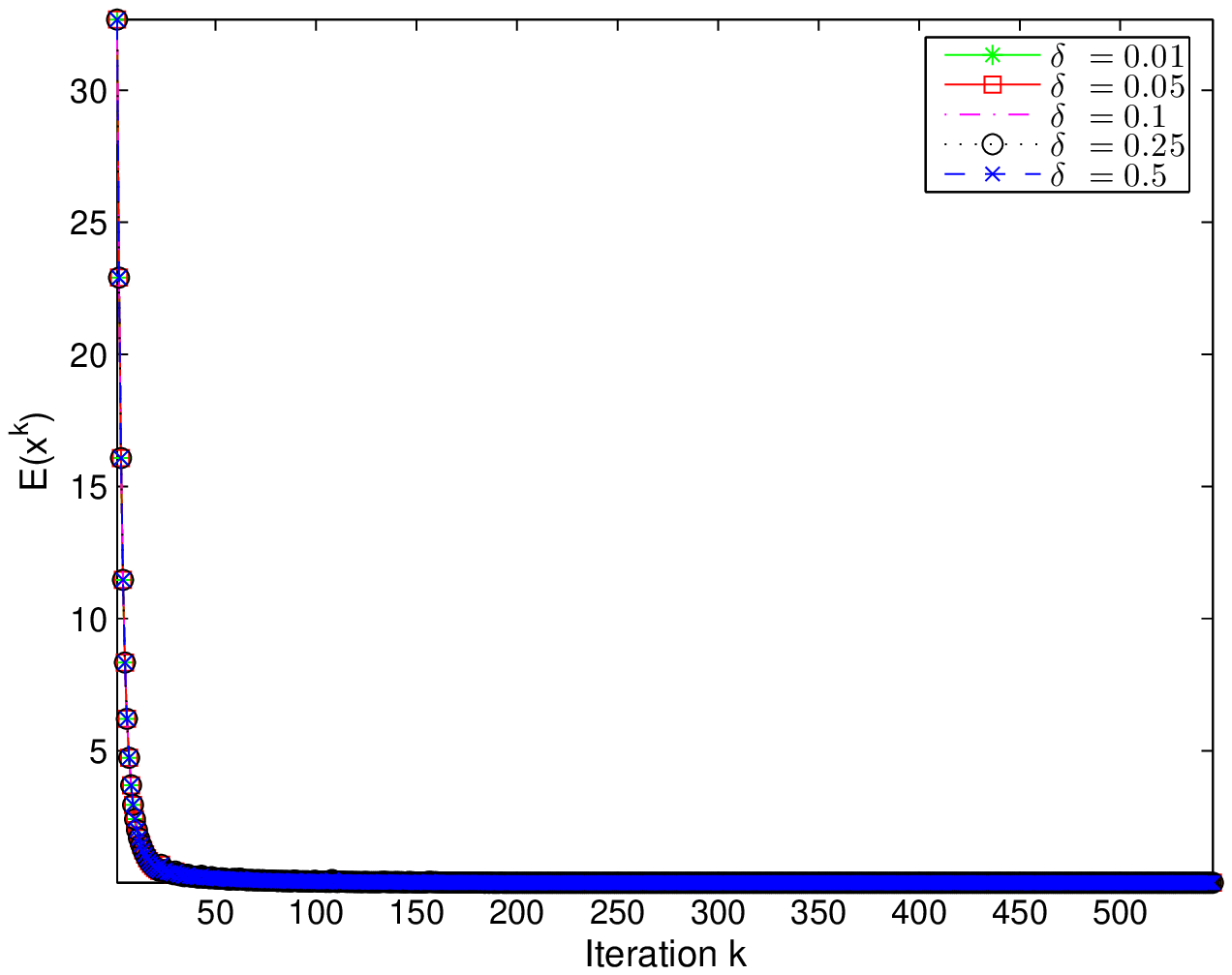}
}
\centerline{{\bf Fig. 5} Results for Algorithm 1 with different $\delta$ ($\theta=0.1$ and $\beta_k=0.5, \forall k \in \mathbb{N}$)}
\label{fig:5}
\end{figure}

Take $\left(x_{\min }^{g}\right)_{j}$  and $\left(x_{\max }^{g}\right)_{j}$ as the lower and upper bounds for the power generating by the unit $j$, respectively. Moreover, take $\left(x_{\min }^{c}\right)_{j}$  and $\left(x_{\max }^{c}\right)_{j}$ to be the lower and upper bounds for the power generating by the company $j$, respectively. All the aforementioned bounds are shown in Table 2. Thus we obtain the strategy set of the problem given by
$$C =\left\{x=\left(x_{1}, \ldots, x_{n^g}\right)^{T} :\left(x_{\min }^{g}\right)_{j} \leq x_{j} \leq\left(x_{\max }^{g}\right)_{j}, \forall j\right\}.$$
We take $q^{i} =\left(q_{1}^{i}, \ldots, q_{n^g}^{i}\right)^{T}$ with
$$q_{j}^{i}=\left\{\begin{array}{ll}{1,} & {\text { if } j \in I_{i}} \\ {0,} & {\text { if } j \notin I_{i},}\end{array}\right.$$
and set
$$\begin{array}{l}{A =2 \sum_{i=1}^{n^{c}}\left(1-q^{i}\right)\left(q^{i}\right)^{T},\quad B =2 \sum_{i=1}^{n^{c}} q^{i}\left(q^{i}\right)^{T}}, \\ {a =-387.4 \sum_{i=1}^{n^{c}} q^{i}, \quad \text { and } \quad c(x) =\sum_{j=1}^{n^{g}} c_{j}\left(x_{j}\right).}\end{array}$$
Consequently, the involved function $f : \mathbb{R}^6\times\mathbb{R}^6 \to\mathbb{R}$ is defined by
$$f(x, y)=[(A+B) x+B y+a]^{T}(y-x)+c(y)-c(x),\quad \forall x, y \in\mathbb{R}^6.$$
Note the fact that $A$ is not positive semidefinite, $B$ is symmetric positive semidefinite and $c(x)$ is a nonsmooth convex function. Combining this fact with the result that $f(x, y)+f(y, x)=-(y-x)^TA(y-x)$ proved in \cite{Dual-extragradient-EP2012},  we obtain that $f$ is nonmonotone and nonsmooth.

We apply Algorithm 1 for solving this problem and compare with PA of Dinh and Kim \cite{projection-NEP2016}. Take $x^0=(20,50,40,45,30,30)^T$ as the initial point in $C$.
The numerical results and the dependency of our Algorithm 1 on the parameters, such as $\{\beta_k\}$, $\theta$ and $\delta$, are reported in Tables 3, 4, 5, 6 and 7 and Figs. 1, 2, 3, 4 and 5. The reports show the strength and efficiency of Algorithm 1 for this example. More precisely, Tables 3 and 4 and Fig. 1 show that Algorithm 1 does well in the convergence rate. Table 5 and 6 and Figs. 2, 3 and 4 illustrate that the strict dependency of the convergence rate of our Algorithm 1 on the sequence $\{\beta_k\}$ and the parameter $\theta$, while Table 7 and Fig. 5 show that the convergence rate of Algorithm 1 is independent of the parameter $\delta$.

Now we apply the following problem to illustrate the efficiency of our Algorithm 2, \cite[Algorithm F]{Burachik-multiVI2019} (denoted by PA-BM) and \cite[Algorithm 2.1]{Yiran-He-NVI2015} (denoted by PA-YH).

{\bf Example 2} Consider the quasimonotone variational inequality proposed by Hadjisavvas and Schaible in \cite{Hadjisavvas1996} (see e.g., \cite{Yiran-He-NVI2015, Burachik-multiVI2019}), where $C=[0,1]\times[0,1]$, $t=(x_1+\sqrt{x_1^2+4x_2})/2$ and $F$ is defined by
$$ F(x_1,x_2)=(-t/(1+t),-1/(1+t))^T.$$
The results for this problem are reported in Tabs. 8, 9 and 10 and Fig. 6. They verify that our Algorithm 2 works well for Example 2 and has an advantage in numerical behavior.

\begin{figure}[!htbp]
\centerline{
\includegraphics[scale=0.65]{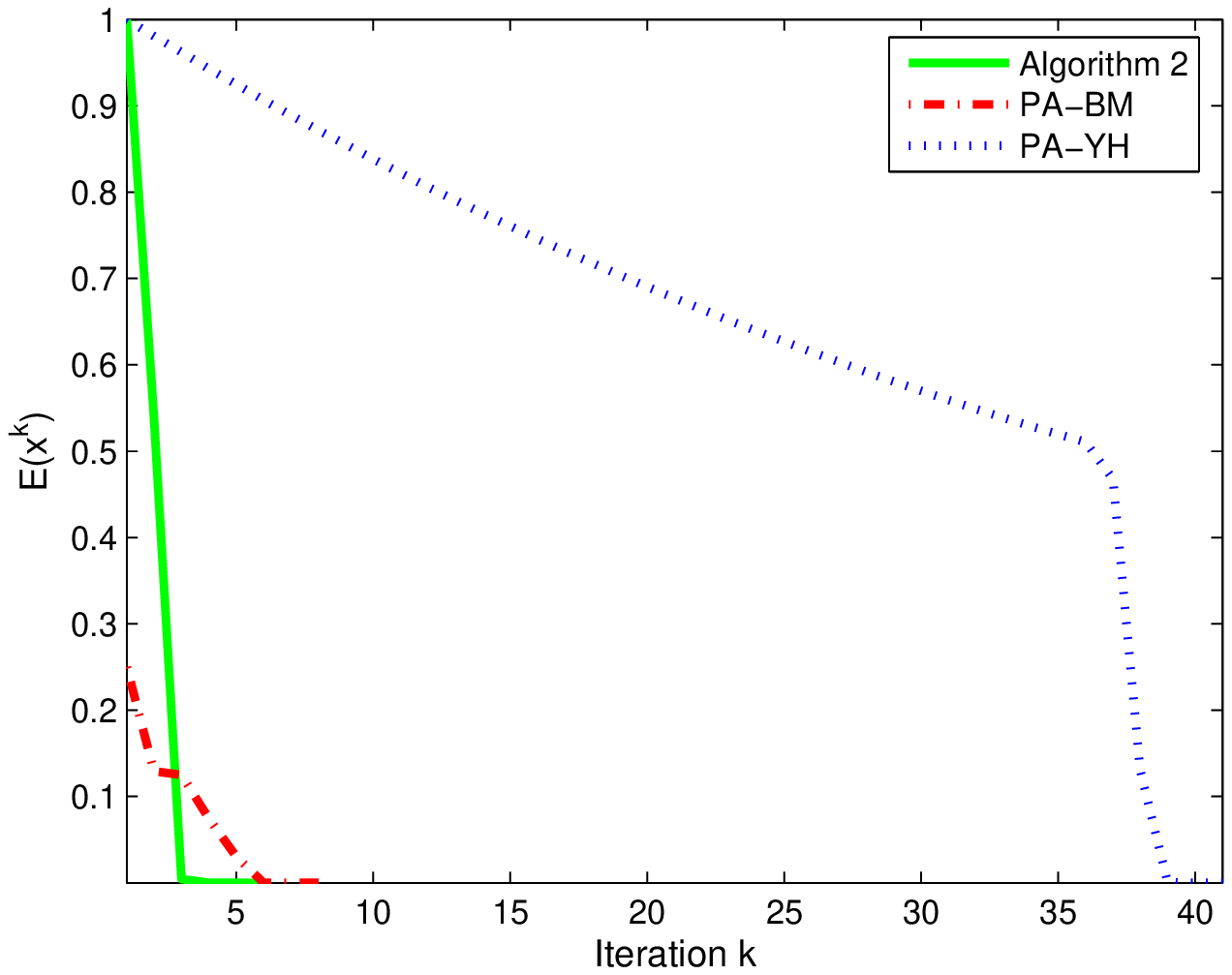}
}
\centerline{{\bf Fig. 6} Results for Algorithm 2, PA-BM and PA-YH with $\delta=0.01$, $\theta=0.95$ and $\beta_k=0.5, \forall k \in \mathbb{N}$}
\label{fig:6}
\end{figure}

\begin{table}[!htbp]
\begin{center}
{\bf Table 8} Results for Algorithm 2, PA-BM and PA-YH with different initial point $x^0$ \\
($\theta=0.95$, $\delta=0.01$ and $\beta_k=0.5, \forall k \in \mathbb{N}$)
\end{center}
\label{tab:8}
\centering
\begin{tabular}{lllllll}
\toprule
\multirow{2}{*}{$x^0$} & \multicolumn{3}{c}{Number of iterations} & \multicolumn{3}{c}{CPU-time(s)} \\
\cmidrule(r){2-4} \cmidrule(r){5-7} \\
&  Algorithm 2     &  PA-BM &  PA-YH
&  Algorithm 2     &  PA-BM &  PA-YH \\
\midrule
 $(0,0)^T$     & 6    &  8    &  41  &   0.7176    &  0.4056   &    2.9796    \\
 $(0,1)^T$     & 5    &  7    &  5   &   0.2808    &  0.2184   &    0.1560   \\
 $(1,0)^T$     & 5    &  8    &  74  &   0.1560    &  0.1872   &    1.4664    \\
 $(1,1)^T$     & 1    &  1    &  1   &   0.0156    &  0.0156   &    0.0156    \\
 $(0.3,0.5)^T$ & 5    &  6    &  36  &   0.0468    &  0.1248   &    0.6708    \\
 $(0.7,0.1)^T$ & 5    &  8    &  74  &   0.1404    &  0.1872   &    1.6536  \\

\bottomrule
\end{tabular}
\end{table}

\begin{table}[!htbp]
\begin{center}
{\bf Table 9} Results for Algorithm 2, PA-BM and PA-YH with different $\beta_k$ \\
($x^0=(0,0)^T$, $\theta=0.5$ and $\delta=0.01$)
\end{center}
\label{tab:9}
\centering
\begin{tabular}{lllllll}
\toprule
\multirow{2}{*}{$\beta_k$} & \multicolumn{3}{c}{Number of iterations} & \multicolumn{3}{c}{CPU-time(s)} \\
\cmidrule(r){2-4} \cmidrule(r){5-7} \\
&  Algorithm 2     &  PA-BM &  PA-YH
&  Algorithm 2     &  PA-BM &  PA-YH \\
\midrule

0.05 & 17  &  106  & 64 &  1.3728   &  5.6316    &  3.1200    \\
0.10 & 17  &  58   & 64 &  0.3588   &  1.3728    &  1.2948    \\
0.15 & 17  &  42   & 64 &  0.3276   &  1.0296    &  1.8252    \\
0.20 & 17  &  34   & 64 &  0.3276   &  0.8736    &  1.4820    \\
0.25 & 17  &  30   & 64 &  0.4836   &  0.5772    &  1.5756    \\
0.30 & 17  &  27   & 64 &  0.4368   &  0.6084    &  1.6536    \\
0.35 & 17  &  24   & 64 &  0.2964   &  0.5148    &  1.9500    \\
0.40 & 17  &  23   & 64 &  0.3900   &  0.5928    &  1.3572    \\
0.45 & 17  &  22   & 64 &  0.3900   &  0.6084    &  1.8408     \\
0.50 & 17  &  21   & 64 &  0.2652   &  0.5772    &  1.7316     \\

\bottomrule
\end{tabular}
\end{table}

\begin{table}[!htbp]
\begin{center}
{\bf Table 10} Results for Algorithm 2, PA-BM and PA-YH with different $\theta$ \\
($x^0=(0,0)^T$, $\delta=0.01$ and $\beta_k=0.5, \forall k \in \mathbb{N}$)
\end{center}
\label{tab:10}
\centering
\begin{tabular}{lllllll}
\toprule
\multirow{2}{*}{$\theta$} & \multicolumn{3}{c}{Number of iterations} & \multicolumn{3}{c}{CPU-time(s)} \\
\cmidrule(r){2-4} \cmidrule(r){5-7} \\
&  Algorithm 2     &  PA-BM &  PA-YH
&  Algorithm 2     &  PA-BM &  PA-YH \\
\midrule

0.05 & 199   &  237   &  383   & 11.3725    & 12.6985   & 20.5609    \\
0.10 & 98    &  117   &  214   & 4.6800     & 6.2400   &  5.7252    \\
0.20 & 47    &  57    &  123   & 1.0452     & 1.4352   &  2.9172    \\
0.25 & 37    &  45    &  104   & 1.1388     & 1.2012   &  2.4024    \\
0.50 & 17    &  21    &  64    & 1.3104     & 1.2324   &  2.7612     \\
0.60 & 13    &  17    &  56    &  0.8580    &  0.4524    &  1.5912    \\
0.70 & 11    &  14    &  50    &  0.2340    &  0.2496    &  1.3728    \\
0.85 & 8     &  10    &  45    &  0.1872    &  0.2808    &  1.4820   \\
0.95 & 6     &  8     &  41    &  0.1716    &  0.1560    &  0.9204  \\
0.99 & 5     &  7     &  57    &  0.0624    &  0.1560    &  1.6536  \\

\bottomrule
\end{tabular}
\end{table}

\section{Conclusion}\label{sec:5}
We propose a projection algorithm with an Armijo-type linesearch for solving nonmonotone and non-Lipschitzian equilibrium problems in Hilbert spaces.  The convergence of the proposed algorithm  requires only the nonemptyness of the solution set of the associated Minty equilibrium problem, instead of  the pseudomonotonicity assumed commonly in many projection-type algorithms. In addition, compared with the existing methods  with same assumptions,  we do not need to employ the Fej\'{e}r monotonicity in the framework of proving the convergence of our algorithm.



\section*{Disclosure statement}

The authors declare that they have no conflict of interest.

\section*{Funding}

This work was partially supported by the National Natural Science  Foundation of China (11471230 and 11771067) and the Applied Basic Research Project of Sichuan Province (2018JY0169).


\end{document}